\documentclass{amsart}
\usepackage[T1]{fontenc}
\usepackage[utf8]{inputenc}
\usepackage{amsmath,amsfonts,amssymb,amsthm}
\usepackage{enumerate}
\usepackage{enumitem}
\usepackage{graphicx}
\usepackage{color}
\usepackage{tikz, tikz-3dplot, pgfplots}
\usetikzlibrary[positioning,patterns] 
\usepackage[all]{xy}
\usepackage[colorlinks,linkcolor=blue,citecolor=red,linktocpage=true]{hyperref} 
\numberwithin{equation}{section}
\numberwithin{table}{section}
\numberwithin{figure}{section}

\title[Symplectic Grassmannians and  rank $2$ symplectic matroids]{On $T$-invariant subvarieties of symplectic Grassmannians  and  representability of rank $2$  symplectic matroids over ${\mathbb C}$}
\author[P. L. del Angel]{Pedro L. del Angel}
\address[P. L. del Angel]{Centro de Investigaci\'on en Matem\'aticas, A. C., Guanajuato, M\'e\-xi\-co.}
\email{luis@cimat.mx}
\author[E. J. Elizondo]{E.~Javier Elizondo}
\address[E. J. Elizondo]{Instituto de Matem\'aticas\\ Universidad Nacional Aut\'onoma de M\'e\-xi\-co\\ 
04510 Ciudad de M\'exico,  M\'exico.}
\email{javier@im.unam.mx}\thanks{The second author would like to thank UNAM and its PASPA program of DGAPA}
\author[C. Garay]{Cristhian Garay}
\address[C. Garay]{Centro de Investigaci\'on en Matem\'aticas, A. C., Guanajuato, M\'exico.}
\email{cristhian.garay@cimat.mx}
\author[F. Zald\'{\i}var]{Felipe Zald\'{\i}var}
\address[F. Zald\'{\i}var]{Departamento de Matem\'aticas\\
Universidad Aut\'onoma Metropolitana-I\\
09340 Ciudad de M\'exico,  M\'exico.}
\email{fz@xanum.uam.mx}

\date{\today}



\DeclareMathOperator{\stab}{stab}
\DeclareMathOperator{\diag}{diag}

\newtheorem{corollary}{Corollary}[section]
\newtheorem{theorem}[corollary]{Theorem}

\newtheorem{lemma}[corollary]{Lemma}
\newtheorem{proposition}[corollary]{Proposition}

\theoremstyle{definition}

\newtheorem{definition}[corollary]{Definition}

\newtheorem{remark}[corollary]{Remark}

\newtheorem{example}[corollary]{Example}

\begin{document}
\maketitle

\begin{abstract}
For the symplectic Grassmannian $\text{SpG}(2,2n)$ of $2$-dimensional isotropic subspaces  in a $2n$-dimensional vector space over an algebraically closed field of characteristic zero  endowed with a symplectic form and with the natural action of an $n$-dimensional torus $T$ on it, we characterize its irreducible $T$-invariant  subvarieties. This characterization is in terms of symplectic Coxeter matroids,  and we use this result to give a complete characterization of the symplectic matroids  of rank $2$ which are representable over $\mathbb{C}$.
\end{abstract}

\section{Introduction}
Given the action of the maximal torus $T$ on a Grassmannian, it  is very hard to characterize the $T$-invariant subvarieties, and despite being an important problem little is known in the literature. In this paper we solve this problem for the symplectic Grassmannian $\text{SpG}(2,2n)$. We are able to characterize all the $T$-invariant subvarieties of this symplectic Grassmannian. 

Combinatorics and algebraic groups are intimately related, and the article \cite{GGMS} of Gelfand, Goresky,  MacPherson and Serganova showed how  matroids are important to understand the $T$-invariant subvarieties. Since not all symplectic matroids of rank two are representable over $\mathbb{C}$ (see \cite[Section 3.4.3]{BGW}), one of the important results in this paper is to characterize which ones are representable.  

Coxeter matroids  play the same role for other types of flag varieties that matroids play for the type-A Grassmannian. Part of their importance rests on
the relations between them and the geometry of the corresponding flag varieties. A natural setting where Coxeter matroids appear is when one considers a semisimple algebraic group $G$ over 
an algebraically closed field $K$ of characteristic zero, a {parabolic subgroup} $P\subseteq G$, a maximal torus $T\subseteq P$  with Lie algebra $\mathfrak{t}$ and  Weyl group $W$, and the projective variety $X=G/P$. Then one has:  
\begin{theorem}[See \cite{BGW}]\label{thm1}
If $\mu:X\longrightarrow\mathfrak{t}^*_\mathbb{R}$ is  the corresponding \textit{moment map}, then  $\mu( \overline{Tx})$ is a  Coxeter matroid associated to the Weyl group $W$ for every $x\in X$.
\end{theorem}

A basic example of this situation is the Grassmannian $G(d,n)$  for $0< d< n$, which is a quotient $G(d,n)=G/P$ for $G=\text{GL}_n(K)$ and  $P=P_{n,d}$ the subgroup of block matrices of the form $\left( \begin{smallmatrix} A&B\\ 0&C \end{smallmatrix}\right)$ where $A\in \text{GL}_d(K)$ and $C\in \text{GL}_{n-d}(K)$. The maximal torus $T\subseteq P$ is the subgroup of diagonal matrices with non-zero entries, and the Weyl group $W$ is the symmetric group $S_n$ on $n$ letters.
 For any $x\in G(d,n)$, the polytope  $\mu( \overline{Tx})$ is then a  Coxeter matroid associated to ${S}_n$ which is representable over $K$; if we denote by 
 $M^d_{S_n}(K)$  the set of these objects, which are just called $K$-representable matroids of rank $d$ on $n$ labels, then we have a decomposition
 \begin{equation}\label{eq1.0}
 G(d,n)=\bigsqcup_{M\in M^d_{S_n}(K)}G_M,
 \end{equation}
 in a  disjoint union of the (nonempty) locally closed sets $G_M=\{x\in G(d,n)\::\:\mu( \overline{Tx})=M\}$. The set $G_M$ is called the {\it (symmetric) thin Schubert cell} of $G(d,n)$ associated to $M\in M^d_{S_n}(K)$. These thin Schubert cells are $T$-invariant spaces that are closely related to the usual Schubert cells and  they fit into a (trivial) torus bundle:
\begin{equation}\label{eq1.1}
    \xymatrix{T_M\ar[r]&G_M\ar[r]&\mathcal{G}_M} \quad \text{for  $M\in M^d_{S_n}(K)$},
\end{equation}
where $T_M\subseteq T$  is a subtorus depending on $M$ acting freely on $G_M$, and $\mathcal{G}_M=G_M/T_M$. Using the quotient spaces $\mathcal{G}_M$ and the action of ${S}_n$ on the set $M^d_{S_n}(K)$, some parameter spaces for the $T$-invariant subvarieties of $G(d,n)$ having a common homology class $\lambda\in H_*(G(d,n),\mathbb{Z})$ were constructed by Elizondo, Fink and Garay in \cite{EFG}.
\medskip

 Our aim now is to apply the program from \cite{EFG} to the symplectic Grassmanniann $\text{SpG}(2,2n)$ of $2$-dimensional isotropic spaces  in a $2n$-dimensional vector space with a symplectic form. We recall that $\text{SpG}(2,2n)$ is as a homogeneous space $\text{Sp}_{2n}/P$ for the symplectic group $\text{Sp}_{2n}$ and $P\subseteq \text{Sp}_{2n}$ a maximal parabolic subgroup. The maximal torus $T\subseteq P$ is a subgroup of certain diagonal matrices with non-zero entries, and the Weyl group is the hyperoctahedral group $BC_n$. 
 
We have that for any $x\in \text{SpG}(2,2n)$  the polytope  $\mu( \overline{Tx})$ is  a  Coxeter matroid associated to ${BC}_n$ which is representable over ${\mathbb C}$; we denote by 
 $M^2_{{BC}_n}({\mathbb C})$  the set of these objects which are just called ${\mathbb C}$-representable ${BC}_n$- matroids of rank $2$ on $2n$ labels. Our proof uses the fact that $\text{SpG}(2,2n)$ is  
 a  hyperplane section of the type-A Grassmannian $G(2,2n)$ (see Section \ref{sec3} for details) and then we apply the theory of type-A (i.e., for the symmetric group) matroids of rank 2, in particular that every type-A matroid is representable  over an algebraically closed field of  characteristic zero (see Remark \ref{remark_representable}). 
 
 Our main results
 include a characterization of the $T$-invariant varieties of the symplectic Grassmannian $\text{SpG}(2,2n)$
 using a moduli of irreducible quotients associated to type-A  and type-BC (symplectic) Coxeter matroids. Using the moment map, we show that for each thin Schubert cell in $\text{SpG}(2,2n)$ there exists a unique lifting to a thin Schubert cell in $G(2,2n)$. As a by-product, we obtain a new characterization of rank 2 symplectic matroids using a type-A representative.

  The article is organized as follows. In Section \ref{sec2} we start by recalling some basic facts about Grassmannians and matroids with emphasis on some relevant results of \cite{EFG}. Next, we fix the notation and recall some facts about symplectic matroids following \cite{BGW} and the action of the maximal $2n$-torus on the Grassmannian $G(2,2n)$. \par
 In Section \ref{sec3}, after recalling its definition as a homogeneous space, we give an explict description of the symplectic Grassmannian $\text{SpG}(2,2n)$  as a hyperplane section of the Grassmannian $G(2,2n)$. This allows us to compute its integral homology group and describe its Schubert cells. Next, we 
  detail the structure of $\text{SpG}(2,2n)$ as a $T$-variety, where $T$ is an $n$-dimensional torus, and calculate its $T$-fixed points, orbits and stabilizers. Theorem \ref{thm3.14} and Proposition \ref{prop3.16} give the symplectic analogue of the torus bundle \eqref{eq1.1}, and  
  Theorem \ref{thm3.18} gives a description of the $T$-invariant subvarieties of the symplectic Grassmannian $\text{SpG}(2,2n)$.\par
In Section \ref{sec4}  we obtain 
in Theorem \ref{theorem3.15}  
a stratification \eqref{eq5.1a} of the symplectic Grassmannian
$\text{SpG}(2,2n)$. Using this stratification we obtain the symplectic analogues of thin Schubert cells, a new characterization of symplectic matroids of rank $2$, and a criterion for the representability over ${\mathbb C}$ of these matroids. In Section \ref{sec5} we calculate the dimension of the symplectic thin Schubert cells and obtain a description of the $T$-invariant schematic points of the symplectic Grassmannian.
\newline

\subsection{Notation.} We work over ${\mathbb C}$, and for $n\geq2$, we denote by $[2n]$ the set $\{1,\ldots,2n\}$ and by $N=\binom{2n}{2}-1=n(2n-1)-1$. 
Let $\mathbb{P}^{N}=\mathbb{P}^{N}_{\mathbb C}$ denote the projective space over ${\mathbb C}$  with projective coordinates $x_{i,j}$ for $\{i,j\}\in \binom{[2n]}{2}$. 
Coxeter matroids with respect to the Weyl group $W$ will be called  $W$-matroids.
 
\section{Preliminaries}\label{sec2}
\subsection{The Grassmannian of lines  $G(2,2n)$}\label{subsec2.1}
In this section we  recall some basic facts about Grassmannians  of lines. Unless otherwise stated, order arguments on $[2n]$ will be made with respect to the usual order: $1<2<\cdots<2n$.

Let $G(2,2n)$ denote the Grassmannian variety of $2$-dimensional vector subspaces in a complex vector space of dimension $2n$, which we consider embedded in $\mathbb{P}^{N}$ by means of the Pl\"ucker map.
The algebraic $2n$-dimensional torus $T'$  of invertible diagonal matrices $\text{diag}(\lambda_i)=\text{diag}(\lambda_1,\ldots,\lambda_{2n})$ with $\lambda_i\in{\mathbb C}^*$  acts on a point $x\in G(2,2n)$ by scaling the columns of a matrix representation of it. Thus,
 if $x$ has Pl\"ucker coordinates $[x_{i,j}]$ and $\text{diag}(\lambda_i)\in T'$, then $\text{diag}(\lambda_i)\cdot x$ has Plucker coordinates $[\lambda_i\lambda_jx_{i,j}]$.

\subsection{$S_{2n}$-Matroids of rank 2}\label{subsec2.2}
Given  $x\in G(2,2n)$ with coordinates $[x_{i,j}]\in \mathbb{P}^{N}$, we denote by ${\mathcal B}'(x)=\{\{i,j\}\in\binom{[2n]}{2}\::\:x_{i,j}\neq0\}$. Let $\mu_{T'}: G(2,2n)\to \mathbb{R}^{2n}$ be the corresponding  moment map. If $\{e_i\::\:1\leq i\leq 2n\}$ is the canonical basis of $\mathbb{R}^{2n}$, then $\overline{T'\cdot x}$ is a (normal) toric variety of dimension $0\leq d\leq 2n-1$ satisfying

\begin{equation}
\label{equation_moment}
    \mu_{T'}(\overline{T'\cdot x})=\text{Conv}\{e_i+e_j\::\:\{i,j\}\in {\mathcal B}'(x)\}.
\end{equation}

Conversely, we can recover the set ${\mathcal B}'(x)$ from the polytope $\mu_{T'}(\overline{T'\cdot x})$ by taking its set of integral points: ${\mathcal B}'(x)=\mu_{T'}(\overline{T'\cdot x})\cap\mathbb{Z}^{2n}$. 
In particular, if $x=[x_{\{i,j\}}]$ satisfies $x_{\{i,j\}}\neq0$ for all $\{i,j\}\in\binom{[2n]}{2}$, then ${\mathcal B}'(x)=\mu_{T'}(G(2,2n))$   is the  hypersimplex 
\begin{equation}
 \Delta(2,2n)=\text{Conv}\Big\{e_1+e_j\::\:\{i,j\}\in\binom{[2n]}{2}\Big\}\subset H
\end{equation}
which is a lattice polytope of maximal dimension $2n-1$ inside the the affine hyperplane $H=\{x\in\mathbb{R}^{2n}\::\:\sum_ix_i=2\}$.

\begin{definition}
\label{S2nmatroid}
 If $x\in G(2,2n)$, we call either the subpolytope $\mu_{T'}(\overline{T'\cdot x})\subset \Delta(2,2n)$ or the pair $([2n],{\mathcal B}'(x))$ the $S_{2n}$-{\it matroid} of $x$, and we denote it by $M'(x)$.
\end{definition}

\begin{remark}
We have that $M'(x)$ is an $S_{2n}$-matroid of rank 2 on $[2n]$, and  every such matroid is of this form. This means that the set $M^2_{S_{2n}}$  of $S_{2n}$-matroids of rank 2 on $[2n]$ coincides with the set $M^2_{S_{2n}}({\mathbb C})$ of $S_{2n}$-matroids of rank 2 representable over ${\mathbb C}$.
\end{remark}

\begin{definition}
\label{def:nif}
 A non-intersecting family of $[2n]$ is a family $\Pi=\{P_1,\ldots,P_l\}$ of $l\geq2$ non-empty, pairwise disjoint subsets of $[2n]$. Sometimes we refer to the sets $P_i$ as the bags of $P$. 
\end{definition}

A matroid of rank 2 on $[2n]$ is a pair $M=([2n],\Pi)$, where $\Pi$ is a non-intersecting family of $E_n$.  The set of bases $B(M)$ of this matroid is the set of partial $2$-transversals of $\Pi$. See \cite{EFG}, where they were called {\it partitions}.
\subsection{The quotient set $M^2_{S_{2n}}/S_{2n}$}\label{subsec2.3}
Following \cite[Section 4]{EFG}, 
given $\tau\in S_{2n}$ and $\{i,j\}\in\binom{[2n]}{2}$, we define $\tau(\{i,j\})=\{\tau(i),\tau(j)\}$. We can define an action of $S_{2n}$ on $M^2_{S_{2n}}$ as follows: if $M=([2n],B)\in M^2_{S_{2n}}$ and $\tau\in S_{2n}$, then $\tau\cdot M:=([2n],\tau(B))$. The resulting orbit space $M^2_{S_{2n}}/S_{2n}$ can be described through partitions as follows.

We denote by $\Pi_{2n}$ the set of formal expressions $\prod_{i=1}^\ell k_i$ such that $\ell\geq2$, $1\leq k_i<2n$, and $w(\pi)=\sum_ik_i\leq 2n$.  Then   $M^2_{S_{2n}}/S_{2n}\cong \Pi_{2n}$ as sets (cf. \cite[Theorem 4.8]{EFG}). Furthermore, there is a canonical section  $\Pi_{2n}\hookrightarrow M^2_{S_{2n}}/S_{2n}$ constructed as follows: for $\pi=\prod_{i=1}^l k_i\in\Pi_{2n}$ with $k_1\geq\cdots\geq k_l$ and $w(\pi)=\sum_ik_i$, we will  form a particular non-intersecting family $\{P_{k_1},\ldots,P_{k_l}\}$ of $[w(\pi)]=\{1,\ldots,w(\pi)\}$, which we will also call $\pi$, defined by:
\begin{equation}
\label{particular_partition}
P_{k_i}=
\begin{cases}
\{1,\ldots,k_1\},&\text{ if }i=1,\\
\{k_1+\cdots+k_{i-1}+1,\ldots,k_1+\cdots+k_{i}\},&\text{ if }i>1.\\
\end{cases}
\end{equation}
We denote by $M(\pi)=([w(\pi)],B_\pi)$ the loopless matroid on $[w(\pi)]$ defined by the partition \eqref{particular_partition}, that is, $\{i,j\}$ is not a basis if and only if $i,j$ belong to the same $P_{k_j}$.

The section $\Pi_{2n}\hookrightarrow M^2_{S_{2n}}/S_{2n}$ sends $\pi$ to  $M_\pi=([n],B_\pi)=M(\pi)\oplus (U_1^0)^{\oplus {[n-w(\pi)]}}$, where $(U_1^0)^{\oplus {[n-w(\pi)]}}$ is the uniform matroid on $[n-w(\pi)]$ of rank 0, i.e., it only consists of loops.

\begin{remark}
\label{remark_representable}
With these observations we can  see that any matroid of rank 2 is representable. If $M$ is a matroid with equivalence class $\pi$, then there exists $\tau\in S_n$ such that $\tau(M)=M_\pi$. It is easy to construct a linear space $V\subset {\mathbb C}^n$ of dimension 2 such that $M(V)=M_\pi$, namely, we choose $\ell(\pi)$ vectors on $\mathcal{M}_{0,\ell(\pi)}$ (the Moduli space of marked rational curves with $\ell(\pi)$ marks) which will be the columns of a matrix $A$ such that $M(A)=M_\pi$.

\end{remark}

\subsection{$BC_{n}$-Matroids of rank 2}\label{subsec2.4}
The source for this part is \cite[Chapter 3]{BGW}.    It is convenient to define  $E_n=[n]\cup [n^*]$, where $[n]=\{1,\ldots,n\}$ and $[n^*]=\{n^*,\ldots,1^*\}$ with $i^*=2n-i+1$ for $1\leq i\leq n$,  which satisfies $i^{**}=i$.  If $S\subset E_n$ denote by ${S}^*=\{i^*\::\:i\in S\}$. We denote the obvious bijection by $p:E_n \rightarrow [2n]$.

The {\it hyperoctahedral group} $BC_n=\{\tau\in S_{2n}\::\:\tau({i}^*)={\tau(i)}^*\}\subset S_{2n}$ is the Weyl group of type B-C, and is isomorphic to the group of symmetries of the $n$-cube $[-1,1]^n\subset \mathbb{R}^n$.

\begin{definition}\label{Def_BC_nM}
(1) We say that $S\subset E_n$ is an {\it admissible set} if $S\cap{S}^*=\emptyset$. We denote by $J_n^2=\{\{i,j\}\in\binom{E_n}{2}\::\:\{i,j\}\text{ is an admissible set}\}$; this is, $\{i,j\}\in\binom{E_n}{2}$ such that $j\neq i^*$.

\noindent (2) A total order $<$ on $E_n$ is {\it admissible} if when the $2n$ elements are listed from largest to smallest, the first $n$ elements form an admissible set, and the last listed $n$ elements  are the stars of the first $n$ elements listed in reverse order.
We will consider the standard admissible ordering 
 $1<2<
\cdots<n<n^*<\cdots<2^*<1^*$ on $E_n$.

\noindent (3) For any admissible order $<$ on $E_n$, we define a partial order on $J_n^2$ as follows. Given $\{i<j\},\{k<l\}\in J_n^2$, we say that  $\{i<j\}<\{k<l\}$ if $i< k$ and $j< l$. These partial orders are precisely the $w$-shifted Bruhat orders for $BC_n$, thus the following definitions makes sense:
\begin{enumerate}
    \item For $B\subset J_n^2$, the {triple $M=(E_n,*,B)$ is a $BC_{n}$-matroid} of rank 2 if it satisfies the maximality property, that is, $B$ has a unique maximal element with respect to every admissible order $<$ on $E_n$. These are also called symplectic matroids of rank 2 on $E_n$.  
We say that $B$ is the set of bases of $M$. We will denote by $M^2_{BC_{n}}$ the set of $BC_{n}$-matroids of rank 2 on $E_n$. 
\item Given $\tau\in BC_n$ and $\{i,j\}\in J_n^2$, we set $\tau(\{i,j\})=\{\tau(i),\tau(j)\}$. We have an action of $BC_n$ on $M^2_{BC_{n}}$ defined for { $M=(E_n,*,B))\in M^2_{BC_{n}}$}  and  $\tau\in BC_n$, by $\tau\cdot M:=(E_n,*,\tau(B))\in M^2_{BC_{n}}$, where by the definition of $W$-matroid we have that $\tau\cdot M\in M^2_{BC_{n}}$.
\end{enumerate}

For the canonical order $<$ on $J_n^2$,  two $BC_{n}$-matroids $B,B'\subset (J_n^2,<)$ are isomorphic if there exists $\tau\in BC_n$ such that $\tau(B)=B'$.
\end{definition}

\begin{remark}
The number of admissible orders on $E_n$ is $2^n\cdot n!$, and half of them are the opposite orders of the other half. The same holds for $J_n^2$.
It follows that the maximality property from Definition \ref{Def_BC_nM} is equivalent to the minimality property: $B$ has a unique minimal element with respect to every admissible order $<$ on $E_n$. Hence, every matroid $B\subset (J_n^2,<)$ is contained in an interval $[a,b]=\{x\in J_n^2\::\: a\leq x\leq b\}$
\end{remark}

\begin{example}
\label{ex:iso_class}
If $n=2$, there are $2^22!=8$ admissible orders, the four different ones are 
\begin{align*}
1<2<2^*<1^*\qquad 1<2^*<2<1^*\\
2<1<1^*<2^*\qquad 2<1^*<1<2^*
\end{align*}
and the other four are the opposite orders to these. Thus, the admissible orders on $J_2^2$ are 
\begin{align*}
12<12^*<21^*<2^*1^*\qquad 12^*<12<2^*1^*<21^*\\ 21<21^*<12^*<1^*2^*\qquad 21^*<21<1^*2^*<12^* 
\end{align*}
and the other four are the opposite orders to these. Since these are total orders, every nonempty subset of $J_2^2$ is a $BC_{2}$-matroid of rank 2 and the set $M^2_{BC_{2}}$ has $15=2^4-1$ elements. 
The group $BC_2$ is the dihedral group $D_8$, and $M^2_{BC_{2}}/BC_2$ has five orbits:
\begin{enumerate}[]
    \item $A_1=\binom{J_2^2}{1}$,
    \item $A_{2,1}=\{\{\{1,2\},\{1^*,2^*\}\},\{\{1^*,2\},\{1,2^*\}\}\}$,
    \item $A_{2,2}=\binom{J_2^2}{2}\setminus A_{2,1}$
    \item $A_3=\binom{J_2^2}{3}$
    \item $A_4=J_2$
\end{enumerate}
Compare with Example \ref{example3.3}.
\end{example}

\section{The Symplectic Grassmannian of isotropic lines}\label{sec3}

Given  a positive integer $n$, let $V$ be complex vector space $V$ of  dimension $2n$ endowed with a non-degenerate alternating  bi-linear form $\omega$. Given another integer $1\leq k\leq n$, an element $W\in G(k,2n)=G(k,V)$ is isotropic if $\omega(W,W)=0$.  As a set,
the {\it symplectic Grassmannian} is 
$$\text{SpG}(k,2n)=\{W\in G(k,2n): W \text{ is isotropic}\}.$$
This is a projective variety that can be described  from different perspectives, which we now recall.

\subsection{As a homogeneous space}\label{subsec3.1}
Choosing an adequate basis of the symplectic vector space $(V,\omega)$, the Gram matrix of the symplectic form is $J=\left(\begin{smallmatrix}0&J_n\\-J_n&0
\end{smallmatrix}\right)$, where $J_n$ is the  $n\times n$ matrix with $1$'s in its antidiagonal. This defines the involution  $\text{GL}_{2n}\xrightarrow\sigma  \text{GL}_{2n}$ as  $A\mapsto J(A^t)^{-1}J^{-1}$, where $J^{-1}=-J$. Hence, 
the symplectic group is $\text{Sp}_{2n}=\{A\in \text{GL}_{2n}\::\:\sigma(A)=A\}$,  The natural transitive action of $\text{GL}_{2n}$ on the Grassmannian $G(k,2n)$ restricts to an action of $\text{Sp}_{2n}$ on the symplectic Grassmannian $\text{SpG}(k,2n)$, which is transitive by the Witt Extension Theorem \cite{lam}.
The Borel group  of $\text{Sp}_{2n}$ is the intersection $B=B'\cap \text{Sp}_{2n}$ 
of the Borel subgroup $B'$  of  $\text{GL}_{2n}$ with $\text{Sp}_{2n}$. Thence, writing the Grassmannian as a homogeneous space
\[
G(k,2n)\cong \text{GL}_{2n}/P',
\]
where $P'=\text{Stab}_{\text{GL}_{2n}}(W)\subseteq \text{GL}_{2n}$ a maximal parabolic subgroup,  the stabilizer of an element $W\in G(k,2n)$, and using that the action of $\text{Sp}_{2n}$ is transitive, we can write
\[
\text{SpG}(k,2n)= \text{Sp}_{2n}/\text{Stab}_{\text{Sp}_{2n}}(W),
\]
where $P=\text{Stab}_{\text{Sp}_{2n}}(W)=P'\cap \text{Sp}_{2n}$. Therefore, as a homogeneous space
\[
\text{SpG}(k,2n)=\text{Sp}_{2n}/P.
\]
It is known that $\text{SpG}(k,2n)$ is a smooth
projective irreducible variety of dimension $k(4n-3k+1)/2$. 
 We have:
\begin{enumerate}
\item The natural map $\text{SpG}(k,2n)\simeq \text{Sp}_{2n}/P\to  \text{GL}_{2n}/P'\simeq G(k,n)$ is a closed embedding.
\item $B\omega P\subset B'\omega P'$ for any $\omega\in \text{Sp}_{2n}$ and moreover, 
\[
y\in B\, x\, P\quad\text{if and only if} \quad y\in B'\, x\, P\quad \text{and}\quad y,x\in \text{Sp}_{2n}.
\]
\item In particular, the  Schubert cells in $\text{SpG}(k,2n)\cong \text{Sp}_{2n}/P$ are restrictions to $\text{Sp}_{2n}/P$ of the  Schubert cells in $\text{GL}_{2n}/P'$, i.e.
\begin{equation}\label{eq3.1.0}
    B \overline{x}  = (B' [x])\cap \text{Sp}_{2n}/P
\end{equation}
for every $\overline{x}\in \text{Sp}_{2n}/P$, where $[x]$ is the class of $x$ in $\text{GL}_{2n}/P'$. The inclusion $\subseteq$ in \eqref{eq3.1.0} is clear, while the other one implies the irreducibility of the intersection.
\end{enumerate}

 \subsection{As a linear section of the Grassmannian}\label{subsec3.2} 
With the notation of Section \ref{subsec3.1}, for $2\leq k\leq n$ considering the $k$-th exterior power $\bigwedge^kV$
and the contraction map $c:\bigwedge^kV\rightarrow \bigwedge^{k-2}V$ given by
$$c(v_1\wedge\cdots\wedge v_k)=\sum_{1\leq r<s\leq k}(-1)^{r+s}\omega(v_r,v_s)v_1\wedge\cdots\wedge \hat{v_r}\wedge\cdots\wedge \hat{v_s}\wedge\cdots\wedge v_k.$$
By \cite{BO}, see also \cite{CPZ}, we have that, as projective varieties
$$\text{SpG}(k,2n)=G(k,2n)\cap{\mathbb P}(\ker c)$$
and the vanishing ideal of $\text{SpG}(k,2n)$ is generated by the Pl\"ucker relations in ${\mathbb P}(\bigwedge^kV)$ that define $G(k,2n)$ and the symplectic linear relations, one for each $S\in\binom{E_n}{k-2}$, of the form 
\begin{equation}\label{eq3.2.0}
  \sum_{i,S}x_{Sii^*} 
\end{equation}
where the sum is over all $i\in [n]$ and $i^*\in[n^*]$, with $E_n=[n]\cup [n^*]$ as in Section \ref{subsec2.4}, such that $\{i,i^*\}\cap S=\emptyset$. Moreover, the embedding $\text{SpG}(k,2n)\hookrightarrow {\mathbb P}(\ker c)$ is nondegenerate. The codimension of $\text{SpG}(k,2n)$ inside $G(k,2n)$ is $\binom{k}{2}$, and the corresponding ideal is generated by $\binom{2n}{k-2}$ polynomials (see \cite[pp. 5-6]{BO}). Therefore, $\text{SpG}(k,2n)$ is a smooth complete intersection only for $k=2$. 

\subsection{The symplectic Grassmannian of isotropic lines}\label{subsec3.3} 
For the particular case when $k=2$ there is only one symplectic linear relation \eqref{eq3.2.0} of the form
\begin{equation}
\label{symplectic_rel}
    s=\sum_{i=1}^nx_{i,i^*},
\end{equation}
which defines a hyperplane ${\mathcal H}=\{s=0\}$ such that
$\text{SpG}(2,2n)=G(2,2n)\cap {\mathcal H}$ is smooth.

\begin{proposition}\label{Prop_homology}
We have 
\begin{equation}
    H_i(\text{\rm SpG}(2,2n),\mathbb{Z})=\begin{cases}
    H_i(G(2,2n),\mathbb{Z}),&\text{if }i\leq 4n-5,\\
    H_{i+2}(G(2,2n),\mathbb{Z}),&\text{if }i\geq 4n-4.
\end{cases}
\end{equation}
In particular, they share the same generators for $i\leq 4n-6$, and for  $i\geq 4n-4$, if $\{[\Sigma_j]\}_j$ is a family of generators for $H_{i}(G(2,2n),\mathbb{Z})$, then $\{[\Sigma_j\cap {\mathcal H}]\}_j$ is a family of generators for $H_i(\text{\rm SpG}(2,2n),\mathbb{Z})$.
\end{proposition}
\begin{proof}
We have that  $\text{Sp}_{2n}/P$ is a smooth complete intersection of complex dimension $2(2n-2)-1=4n-5$. 
By Lefschetz hyperplane section  theorem, 
\begin{equation}\label{eq3.5}
H^i(\text{Sp}_{2n}/P,\mathbb{Q})\simeq H^i(\text{GL}_{2n}/P',\mathbb{Q})\quad\text{for }\; i<4n-5 
\end{equation}
and for $k=4n-5$ both sides are trivial. After that, we can apply Poincaré duality to compute its 
remaining homology groups.
\end{proof}

\begin{example}
We have that $H_i(\text{SpG}(2,4),\mathbb{Z})=\mathbb{Z}$ for $i=0,2,4,6$ and zero otherwise. The generators are $[pt.]$ for $i=0$, $[\mathbb{P}^1]$ for $i=2$, $[H\cap Z(s)]$ for $i=4$, where $H\subset G(2,4)$ is a hyperplane  section, and $[\text{SpG}(2,4)]=[G(2,4)\cap Z(s)]$ for $i=6$.
\end{example}

 In general, we have explicit generators of $H_i(\text{SpG}(2,2n),\mathbb{Z})$ using \cite[Theorem 5.4]{EFG}. The homology classes of the of the  Schubert varieties $[\Sigma]=[\Sigma_{\{i,j\}}]$ in $H_*(G(2,2n),\mathbb{Z})$ are indexed by elements $\{i,j\}\in\binom{[2n]}{2}$, where $\Sigma_{\{i,j\}}\subset G(2,2n)$ can be constructed by considering the poset $(\binom{[2n]}{2},\leq)$ and the order is induced by the 
canonical admissible order $\leq$ on $[2n]$. Then the ideal of  $\Sigma_{\{i,j\}}$ in $G(2,2n)$ is generated by $\{x_{\{k,l\}}\::\:\{k,l\}\in\binom{[2n]}{2},\text{ and }\{k,l\}\nleq \{i,j\}\}$.

Another way to say this is that if $x\in \Sigma_{\{i,j\}}$ is generic, then the set of bases ${\mathcal B}'(x)$ of the matroid associated to $x$ (see Section \ref{subsec2.2}) becomes ${\mathcal B}'(x)=\{\{k,l\}\in\binom{[2n]}{2}\::\:\{k,l\}\leq \{i,j\}\}$. Recall that if we define $r:\binom{[2n]}{2}\xrightarrow[]{}\mathbb{N}$ as $r(\{i,j\})=i+j-3$, then $\text{dim}_{\mathbb{C}}(\Sigma_{\{i,j\}})=r(\{i,j\})$, and $(\binom{[2n]}{2},\leq,r)$ becomes a ranked poset. As for the classical case, we have equations for the Schubert varieties in the symplectic case.

\begin{proposition}
For any $\{i,j\}$ in  the poset $(J_n^2,\leq)$, we define 
$${\rm Sp}\Sigma_{\{i,j\}}=Z(\{x_{\{k,l\}}\::\:\{k,l\}\in J_n^2,\text{ and } \{k,l\}\nleq \{i,j\}\})\subset \text{\rm SpG}(2,2n).$$
These are the Schubert varieties whose homology classes $[{\rm Sp}\Sigma]=[{\rm Sp}\Sigma_{\{i,j\}}]$ generate $H_*(\text{\rm SpG}(2,2n),\mathbb{Z})$.
\end{proposition}
\begin{proof}
Using that  $\text{SpG}(2,2n)=G(2,2n)\cap {\mathcal H}$  and the fact that $(J_n^2,\leq)$ is a sub-poset of $(\binom{[2n]}{2},\leq)$, the result follows at once from Equation \eqref{eq3.1.0}.
\end{proof}

\begin{proposition}
Let $\{i,j\}\in J_n^2$, then 
\begin{equation}
    \dim_{\mathbb{C}}({\rm Sp}\Sigma_{\{i,j\}})=\begin{cases}
    p(i)+p(j)-3,&\text{if }p(i)+p(j)<2n+1\\
    p(i)+p(j)-4,&\text{if }p(i)+p(j)>2n+1,
    \end{cases}
 \end{equation}   
 where $p:E_n\rightarrow [2n]$ is the bijection of Section  {\rm \ref{subsec2.4}}.

\end{proposition}

\begin{proof}
Given $\{i,j\}\in\binom{[2n]}{2}$ with the canonical order, we know that $\text{dim}_{\mathbb{C}}(\Sigma_{\{i,j\}})=i+j-3$. Also, given $\{i,j\}\in J_n^2$ we know by Proposition \ref{Prop_homology}  that 
\begin{equation}
   \dim_{\mathbb{C}}({\rm Sp}\Sigma_{\{i,j\}})=\begin{cases}
   \dim_{\mathbb{C}}(\Sigma_{\{p(i),p(j)\}}),&\text{if }p(i)+p(j)<2n+1\\
   \dim_{\mathbb{C}}(\Sigma_{\{p(i),p(j)\}})-1,&\text{if }p(i)+p(j)>2n+1
    \end{cases}
\end{equation}
which ends the proof.
\end{proof}

\subsection{The Symplectic Grassmannian as a $T$-variety}
\label{sect_3.4}
For the maximal torus $T'$ of diagonal invertible matrices in the group $\text{GL}_{2n}({\mathbb C})$ and the involution $\sigma:\text{GL}_{2n}({\mathbb C})\rightarrow \text{GL}_{2n}({\mathbb C})$ of Section \ref{subsec3.1}, the fixed-point subgroup $T=T'^{\sigma}\subseteq T'$ is a maximal $n$-dimensional torus in $\text{Sp}_{2n}$ consisting of diagonal invertible matrices of the form
$\text{diag}(\mu_1,\ldots,\mu_n,\mu_n^{-1},\ldots, \mu_1^{-1})$. The natural action of $2n$-torus $T'=({\mathbb C}^*)^{2n}$ on the Grassmannian $G(2,2n)$ of Section \ref{subsec2.1} restricts to an action of the symplectic $n$-torus $T$ on the symplectic Grassmannian.
Explicitly, if  $x\in G(2,2n)$ we consider a matrix $A$ of size $2n\times 2$ representing $x$. Writing 
$$t=\text{diag}(t_1,\ldots,t_{2n})=\text{diag}(\mu_1,\ldots,\mu_n,\mu_{n^*},\ldots,\mu_{1^*}),$$
where $\mu_{i^*}=\mu_i^{-1}$ for $1\leq i\leq n$,
then
\begin{equation}
\label{equa_action_Tn}
  t\cdot A= \text{diag}(\mu_1,\ldots,\mu_n,\mu_{n^*},\ldots,\mu_{1^*})\cdot A.
\end{equation}
In terms of Pl\"ucker coordinates,  the action \eqref{equa_action_Tn} becomes
\begin{equation}\label{torus_action}
    t\cdot[x_{i,j}]=
    [t_it_jx_{i,j}],\;
    \text{where } t_i=\begin{cases}\mu_i  & \text{for }  i\in [n],\\  
    \mu_{i^*}   & \text{for }  i\in [n^*].\end{cases}
\end{equation}
It is easy to check that
 the $T$-fixed points of $\text{SpG}(2,2n)$ have the form $[0,\ldots,1,\ldots, 0]$ in $\text{SpG}(2,2n)$. Explicitely, by  Equation
 \eqref{symplectic_rel}, these are the
$2n(n-1)$ points $[p_{i,j}]$ having Pl\"ucker coordinate $1$ at $x_{i,j}$ with $\{i,j\}\in J_2^n$.

The inclusion map  $i:T\hookrightarrow T'$ of tori  splits by the map $j:T'\rightarrow T$  given by $j(\text{diag}(\lambda_1,\ldots,\lambda_{2n})=\text{diag}(\lambda_1,\ldots,\lambda_n,\lambda_n^{-1},\ldots,\lambda_1^{-1})$. 

{From now on, we will use freely the bijection $p:E_n\rightarrow [2n]$ from Section  {\rm \ref{subsec2.4}}.}

Given  $x\in \text{SpG}(2,2n)$ with coordinates $[x_{i,j}]\in \mathbb{P}^{N}$, we define by  ${\mathcal B}(x)$ the set $\{\{i,j\}\in\binom{[E_n]}{2}\::\:x_{i,j}\neq0,\: \{i,j\}\in J_2^n\}$, which is the analogous to ${\mathcal B}'(x)$  in Section \ref{subsec2.2}. Let $\mu_T: \text{SpG}(2,2n)\to \mathbb{R}^{n}$ be the corresponding moment map. If $\{\phi_i\::\:1\leq i\leq n\}$ is the canonical basis of $\mathbb{R}^{2n}$, we set $\varepsilon_i=\phi_i$ if $i\in [n]$ and $\varepsilon_i=-\phi_i$ if $i\in [n^*]$.

Now, $\overline{T_{}x}$ is a toric variety of dimension $0\leq d\leq n$ which by Equation \eqref{equa_action_Tn} satisfies

\begin{equation}
\label{equation_moment_sym}
    \mu_T(\overline{T_{}x})=\text{Conv}\{\phi_{i,j}=\varepsilon_i+\varepsilon_j\::\:\{i,j\}\in  {\mathcal B}(x)\}.
\end{equation}

By the theory of Gelfan'd-Serganova, the polytope \eqref{equation_moment_sym} is a Coxeter matroid of rank 2 on $E_n$ corresponding to {the hyperoctahedral group} $BC_n$. 

 In particular, if $x\in \text{SpG}(2,2n)$ is a generic point, i.e., one which satisfies $x_{i,j}\neq0$ for all $\{i,j\}\in\binom{[2n]}{2}$, then $\mu_T(\overline{T_{n}x})=\mu_T(\text{SpG}(2,2n))$   is the  polytope 

\begin{equation}
    \text{Sp}\Delta(2,2n)=\text{Conv}\{\varepsilon_{i,j}\::\:\{i,j\}\in J_n^2\}\subset\mathbb{R}^n.
\end{equation}
which is a lattice polytope of maximal dimension $n$.

\begin{definition}
\label{bcnmatroid}
 If $x\in \text{SpG}(2,2n)$, we call either the subpolytope $\mu_T(\overline{Tx})\subset \text{Sp}\Delta(2,2n)$ or the triple {$(E_n,*,\mathcal{B}(x))$} the $BC_n$-matroid of $x$, and we denote it by $M(x)$.
\end{definition}

\begin{example}[The case $n=2$]
\label{example3.3}

We have $1<2< 2^*< 1^*$, and the set $\text{Sp}\Delta(2,4)$ analogue of the octahedron for the case of $G(2,4)$ is the square:

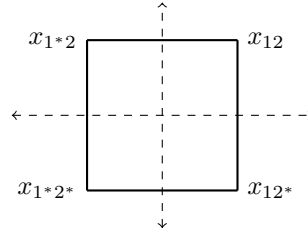
\begin{center}
\begin{tikzpicture}[scale=1]	 
	\draw[dashed,<->] (-2,0) -- (2,0) ; 
	\draw[dashed,<->] (0,-1.5) -- (0,1.5) ; 
	\draw[thick,-] (1,1)-- (-1,1) node[anchor= east]{$x_{1^*2}$};
	\draw[thick,-] (-1,1)-- (-1,-1) node[anchor= east]{$x_{1^*2^*}$};
         \draw[thick,-] (-1,-1) -- (1,-1)   node[anchor= west]{$x_{12^*}$};
	\draw[thick,-]  (1,-1) -- (1,1)   node[anchor= west]{$x_{12}$};	
\end{tikzpicture}
\end{center}

The set of $T$-fixed points is $\{x_{12},x_{1^*2}, x_{12^*}, x_{1^*2^*}\}$, and we may have either 1,2,3 or 4 bases, and in total we have 15 possibilities. Compare with Example \ref{ex:iso_class}.

\begin{table}[!htb]
    \centering
    \begin{tabular}{c|c}
     number of bases& number of $BC_n$ matroids \\
     1&4\\
     2&6 \\
     3&4 \\
     4&1\\
\end{tabular}
    \caption{For $n=2$, there are fifteen $BC_n$-matroids.}
    \label{table3.1}
\end{table}
\end{example}

As in the case for $G(2,2n)$, let $M^2_{BC_{n}}({\mathbb C})$ be the set of $BC_{n}$-matroids of rank 2 on $E_n$ which are representable over ${\mathbb C}$. Contrary to the $S_n$ case,  not all $BC_n$ matroids are representable over ${\mathbb C}$, as we pointed above. 

\begin{definition}
For $N\in M^2_{BC_{n}}(\mathbb{C})$, we denote by 
$$\text{SpG}_N=\{x\in \text{SpG}(2,2n)\::\:M(x)=N\}.$$ The set $\text{SpG}_N$ is non-empty, and it is called the {\it symplectic thin Schubert cell} of $\text{SpG}(2,2n)$ associated to $N$.
\end{definition}

\begin{lemma}\label{lemma3.7}
If $M$ is a type A matroid representable over $\mathbb{C}$, then the stabilizer of any element, in the thin Schubert cell that it defines under the action of the torus $T'$, only depends on $M$.
\end{lemma}

\begin{proof}  
Since the only Plücker coordinates different from zero for the elements in $G_M$  are those corresponding to $\{i,j\}\in M$, all coordinates in $T'$ that do not involve $\{i,j\}\in M$ can be chosen arbitrary and will have no effect on the points in $G_M$. On the other hand, since the action of a diagonal matrix $t$ on $G_M$ correspond to multiplication of each non zero coordinate $x_{ij}$ by $t_i\,t_j$, then $t$ stabilizes a point in
$G_M$ if and only if there exist a unique $\lambda\in\mathbb C^*$ such that
\[
t_i\,t_j = \lambda\quad\quad\text{ for all}\quad \{i,j\}\in M.
\]
Since this condition does not depend on the chosen point but only on $M$, the stabilizer  only depends on $M$ as well and we will denote it simply by $\stab_{T'}(M)$.
\end{proof}

\begin{remark}
If $M$ is a  type BC  matroid of rank 2 representable over $\mathbb{C}$,
Lemma \ref{lemma3.7} may be false as shown in Example \ref{ejemplo:stab}. However, there is a slightly weaker result which is good enough for our purposes, see Lemma \ref{simplecticStab}.
\end{remark}

\begin{lemma} \label{lem:vacuum}
Let $M=(E_n,B)$ be a matroid, $G_M\subset G(2,2n)$ be the thin Schubert cell associated to it and let $\mathcal{H}=Z(s)$ for $s$ as in \eqref{symplectic_rel}.
\begin{enumerate}[label=\emph{(\arabic*)}]
\item If there is no pair $\{j,j^*\}\in B$, then $G_M\subset \mathcal{H}$.
\item If there is only one pair $\{j,j^*\}\in B$, then $G_M\cap \mathcal{H} = \emptyset$.
\item If there are at least two pairs of the form $\{j,j^*\}\in B$, then $G_M\not\subset \mathcal{H}$, 
$G_M\cap \mathcal{H}\ne\emptyset$ and for any $x\in G_M$, we have that $T'x\cap \mathcal{H}$ is a hypersurface on $T'x$ and therefore it is irreducible.
\end{enumerate}
\end{lemma}

\begin{proof} (1) By definition, the only coordinates that are different from zero for any point in $G_M$ are precisely those corresponding to pairs $\{i,j\}\in B$, therefore if no $\{j,j^*\}$ is in $B$, all corresponding coordinates in the equation $s$ are zero and so is their sum.

(2) By the same reason, if only one such coordinate is in $B$, that coordinate, and only that coordinate among the $x_{j,j^*}$ is different from zero, so the sum can not be zero and the intersection is empty.

(3) Finally, if there are at least two such pairs in $B$, then it is always possible to find a point in $G_M$ where their sum is zero, since the only restriction imposed by $M$ is that they are not zero each.

Now let $x=(a_{i,j})\in G_M$ and $\lambda=\diag(\lambda_j)\in T'$. Then $\lambda\cdot x\in G_M\cap \mathcal{H}$
if and only if
\[
\sum_{\{j,j^*\}\in M} \lambda_j\, \lambda_{j^*}a_{j,j^*} = 0.
\]
which defines a polynomial equation on the $\lambda_j$'s and the $\lambda_{j^*}$'s and therefore 
$T'x\cap \mathcal{H}$ is isomorphic to the hypersurface on $T'$ defined by this equation. In particular it is irreducible.\hskip0.5cm{} 
\end{proof}

\begin{lemma}
\label{lemma:existence_and_irred}
  If $G_M\cap \mathcal{H}$ is not empty, then the categorical quotient
${\mathcal G}_{M,{\mathcal H}}=\left(G_M\cap \mathcal{H}\right)/T_{M}$ exists and  is irreducible  
\end{lemma}

\begin{proof}
First we show the existence of the space ${\mathcal G}_{M,{\mathcal H}}$, assuming that it  is irreducible. This follows from  Lemma \ref{lemma3.7}, which says that the torus $T_{M}:=T/\stab_{T}(M)$ acts freely on $G_M\cap \mathcal{H}$.

Now we show its irreducibility. If $G_M\subset \mathcal{H}$ there is nothing to prove. On the other hand, if $G_M\not\subset \mathcal{H}$ but
$G_M\cap \mathcal{H}\ne\emptyset$, then by the lemma above, the restriction $\pi:G_M\cap \mathcal{H}\to\mathcal G_M$ of the fibration $T_{M}\to G_M\to {\mathcal G}_{M}$ from \eqref{eq1.1}  is still surjective with fiber at $\pi(x)$ given by $T'x\cap \mathcal{H}$, which is irreducible by the previous Lemma. Since $\mathcal G_M$ is irreducible, $G_M\cap \mathcal{H}$ is also irreducible. \hskip0.5cm{} 
\end{proof}

\begin{theorem}
\label{thm3.14}
Let $G_M$ be the thin Schubert cell defined by the $S_{2n}$-matroid $M=([2n],B)$. If $G_M\cap \mathcal{H}$ is not empty, then there is a fibration
\[
T_{M} \to (G_M\cap \mathcal H)\to {\mathcal G}_{M,{\mathcal H}},
\]
satisfying 
\begin{equation*}
    \dim {\mathcal G}_{M,{\mathcal H}} =\begin{cases}
        \dim \mathcal G_M +\dim T_M' - \dim T_{M},&\text{ if } B \text{ has no pair }\{i,i^*\},\\
        &\\
        \dim \mathcal G_M  - 1+ \dim T_M' - \dim T_{M},&\text{ if } B \text{ has at least 2 pairs }\{i,i^*\}.\\
    \end{cases}
\end{equation*}
\end{theorem}

\begin{proof} 
As indicated at the beginning of this section, the natural action of $2n$-torus $T'=({\mathbb C}^*)^{2n}$ on the Grassmannian $G(2,2n)$  restricts to an action of the symplectic $n$-torus $T$ on the symplectic Grassmannian. Now  both claims follow at once by applying Lemma \ref{lem:vacuum} and Lemma \ref{lemma:existence_and_irred} to the fibration $T_{M}\to G_M\to {\mathcal G}_{M}$ from \eqref{eq1.1}.
\end{proof}

\begin{proposition} 
\label{prop3.16}
Let $G_M$ be the thin Schubert cell defined by the $S_{2n}$-matroid $M=([2n],B)$. If $G_M\cap \mathcal{H}$ is not empty, then there is a surjective map
$
\pi: {\mathcal G}_{M,{\mathcal H}}\to\mathcal G_M,
$
whose fiber at $[x]\in \mathcal G_M$ is $(T'x\cap \mathcal H)/T_{M}$.
\end{proposition}

\begin{proof} We  only work out the case when $G_M\cap \mathcal H\subset G_M$ is a hypersurface, since the other case is simpler.
 
For any $\overline{y}\in {\mathcal G}_{M,{\mathcal H}}$ define $\pi(\overline{y})=[y]\in \mathcal G_M$. Remember that for any $x\in G_M$, we have that $T'x\cap \mathcal H$ is a non-empty hypersurface by Lemma \ref{lem:vacuum}, therefore the map is surjective.

Next, for any $x\in G_M$ we define an action $T_{M}\times T'x\to T'x$ via $(\lambda, tx)\mapsto (\lambda t)x$. Since
$\mathcal H$ is $T_{M}$ invariant, the action above restricts to a natural free action  $T_{M}\times (T'x\cap \mathcal H)\to(T'x\cap \mathcal H)$.

Now assume $x,y\in G_M\cap \mathcal H$ are such that $\overline{y}\ne\overline{x}$ but $[y]=[x]$, then there exist $t\in T'\backslash T$ such that $y=tx$. Since $x,y \in G_M\cap \mathcal H$, then they are both in $T' x\cap \mathcal H$, but $T_{M}$ acts on $T'x \cap \mathcal H$ and by hypothesis $T_{M}\cdot y\ne T_{M}\cdot x$, so $x$ and $y$ belong to different classes on $(T'x\cap \mathcal H)/T_{M}$. The reciprocal is direct.
\end{proof}

If $G_M\subset \mathcal{H}$, then $T' x\cap \mathcal H= T' x \cong T'_M$, and therefore ${\mathcal G}_{M,{\mathcal H}}=\mathcal G_M\times (T'_M/T_M)$.

If $G_M\cap \mathcal H\subset G_M$ is a hypersurface, then 
$T'x \cap \mathcal H\subset T'$ is isomorphic to a hypersurface of $T_M'$ and different $x\in G_M$ produce isomorphic hypersurfaces on $T'_M$.

\section{On the Representability of Coxeter Matroids of Type $BC_n$}\label{sec4}

Let $\mu_{T'}:G(2,2n)\xrightarrow[]{}\mathbb{R}^{2n}$ be the moment map for the action of torus $T'$ on the Grassmannian $G(2,2n)$. If we consider $\mathbb{R}^{2n}$ with basis $\{e_i\::\:i\in E_n=[n]\cup [n^*]\}$, the image  of the moment map is contained in the polytope $\Delta(2,2n)=\text{Conv}\big\{e_{\{i,j\}}=e_i+e_j\::\:\{i,j\}\in\binom{E_n}{2}\big\}$.

For  $x\in G(2,2n)$ we denote by $M'(x)$ the image of $\overline{T'\cdot x}$ under $\mu_{T'}$. 
This is a Coxeter $S_{2n}$-Coxeter matroid  and it is the lattice subpolytope $\text{Conv}\big\{e_{\{i,j\}}\::\:\{i,j\}\in\binom{E_n}{2},\:x_{i,j}\neq 0 \big\}$ of $\Delta(2,2n)$.

Let $\mu_{T}:\text{SpG}(2,2n)\xrightarrow[]{}\mathbb{R}^n$ be the moment map for the torus $T$ and its action on $\text{SpG}(2,2n)$.  If $\mathbb{R}^{n}$ has basis $\{\phi_i\::\:i\in [n]\}$, the image of the moment map is contained in the polytope $Sp\Delta(2,2n)=\text{Conv}\big\{\phi_{\{i,j\}}=\varepsilon_i+\varepsilon_j\::\:\{i,j\}\in J_2^n\}$, where $\varepsilon_i=\phi_i$ if $i\in[n]$ and $\varepsilon_i=-\phi_i$ if $i\in[n^*]$.
For  $x\in \text{SpG}(2,2n)$ we denote by $M(x)$ the image of $\overline{T\cdot x}$ under $\mu_{T}$. 
This is a Coxeter $BC_n$-matroid and it is the lattice subpolytope 
$\text{Conv}\big\{\phi_{\{i,j\}}\::\:\{i,j\}\in J_2^n,\:x_{i,j}\neq 0 \big\}$  of $Sp\Delta(2,2n)$.

\begin{remark}\label{rem4.1}
For $\mathbb{R}^{2n}$ with basis $\{e_i\::\:i\in E_n\}$, and  $\mathbb{R}^{n}$ with basis $\{\phi_1,\ldots,\phi_{n}\}$ denote by $\pi:\mathbb{R}^{2n}\xrightarrow[]{}\mathbb{R}^{n}$ the projection sending $e_i$ to $\varepsilon_i$.  Note that $\text{Ker}(\pi)$ is spanned by $\{e_i+e_{i^*}\::\:i=1,\ldots,n\}$.

We define $\log\circ|\cdot|:T'\rightarrow {\mathbb R}^{2n}$ by sending 
 $u'=\text{diag}(\lambda_i)\in T'$ to $\log(|u'|)=(\log|\lambda_i|)$, and we define $\log\circ|\cdot|:T\rightarrow {\mathbb R}^{n}$ similarly.
 
Let $j:T'\rightarrow T$ be the map splitting the inclusion $i:T\hookrightarrow T'$ (see Section \ref{sect_3.4}).  Then, the following diagram commutes:

 
\begin{equation}\label{eq4.1.1}
    \xymatrix{T'\ar[r]^{j}\ar[d]_ {\log\circ|\cdot|}& T\ar[d]^ {\log\circ|\cdot|}\\
    \mathbb{R}^{2n}\ar[r]^{\pi}&\mathbb{R}^n\\}
\end{equation}

Since the maps $\log\circ|\cdot|$ are usually called the Archimedean tropicalization maps, we may consider $\pi$ as the Archimedean tropicalization of $j$. 
\end{remark}

The following proposition gives the
relationship between $BC_n$ and $S_{2n}$ matroids:
\begin{proposition}\label{pro4.2}
Let $\pi:\mathbb{R}^{2n}\xrightarrow[]{}\mathbb{R}^{n}$ be the map from {\rm Remark \ref{rem4.1}}. Given $x\in \text{\rm SpG}(2,2n)$, then 
\begin{equation*}
    M(x)=\pi(M'(x))
\end{equation*}
\end{proposition}

\begin{proof}
Since $\pi(e_i)=\varepsilon_i$ for $i\in E_n$, by the linearity of $\pi$  we have 
$$\pi(e_{\{i,j\}})=\pi(e_i+e_j)=\varepsilon_i+\varepsilon_j =\phi_{\{i,j\}}.$$
It follows that $\pi(\{e_{\{i,j\}}\::\:\{i,j\}\in\binom{E_n}{2},\:x_{i,j}\neq 0\})=\{\phi_{\{i,j\}}\::\:\{i,j\}\in J_2^n,\:x_{i,j}\neq0\}$. Lastly, since $\pi$ is linear it commutes with the operation Conv, and the result follows from the definitions of
 $M'(x)=\text{Conv}\big\{e_{\{i,j\}}\::\:\{i,j\}\in\binom{E_n}{2},\:x_{i,j}\neq 0\big\}$ and  $M(x)=\text{Conv}\big\{\phi_{\{i,j\}}\::\:\{i,j\}\in J_2^n,\:x_{i,j}\neq 0 \big\}$.
\end{proof}

\begin{corollary}
\label{cor:sym_lift}
    If $\Delta_N$ is a symplectic matroid which is representable over $\mathbb{C}$, then $\pi^{-1}(\Delta_N)=\{\Delta_M\in M^2_{S_{2n}}\::\:\pi(\Delta_M)=\Delta_N\}$ is non-empty
\end{corollary}
\begin{proof}
    If $\Delta_N=\mu_T(x)$ for some $x\in \text{\rm SpG}(2,2n)$, then $M'(x)=\Delta\in \pi^{-1}(\Delta_N)$. 
\end{proof}

If $M$ is a (symmetric) matroid of rank 2 on the set of labels $E$, it can be described as a pair $M=(E,\Pi)$ where   $\Pi=\{P_1,\ldots,P_\ell\}$ is a non-intersecting family on $E$ with $\ell\geq2$ (see Definition \ref{def:nif}).

\begin{definition}
Given a  matroid $M=(E_n,\Pi)$ in the form of a non-intersecting family, we define its degree $\deg(M)$ as $\deg(M)=\#\{i\::\:i\in P_a,\:i^*\in P_b,\:a\neq b\}$. 
\end{definition}

In \cite{EFG}, there are two operations of codimension 1 that one can perform on $\Pi$:
\begin{enumerate}
\item {\it Erasing} a label $i\in P_a$, which erases any base containing it,
\item {\it Merging} two bags $P_a,P_b\in \Pi$, which we denote by $\mu_{a\cup b}$, i.e. $\mu_{a\cup b}(P_a,P_b)=P_a\cup P_b$, and 
\begin{equation}
\mu_{a\cup b}(\Pi)=\Pi-P_a-P_b+P_a\cup P_b.
\end{equation}
This is the simplest way to erase the base $\{i,j\}$ when $i\in P_a$ and $j\in P_b$. The {\it inverse} operation is to {\it split} a bag $P_a$ of cardinality at least 2 at some element $i\in P_a$, which we denote by $\sigma_{P_a|i}$
 and given by

\begin{equation}
\sigma_{P_a| i}(\Pi)=\Pi-P_a+(P_{a}-\{i\})+\{i\}
\end{equation}
\end{enumerate}

 We now introduce a combinatorial version of the projection $\pi$ from Remark \ref{rem4.1}, which we denote it also by $\pi$.

\begin{definition}
\label{def:sympl_proj}
    Given a matroid $M$ on $E_n$ with bases $B(M)=\{\{i,j\}\::\:i,j\in E_n\}$, let $A(M)=\{i\in E_n: \{i,i^*\}\in B(M)\}$. We define its symplectic projection of $M$ as  $\pi(M)=(E_n,*,B(\pi(M)))$, where  $B(\pi(M))=B(M)\setminus\{\{i,i^*\}\::\:i\in A(M)\}$ are the admissible bases. 
\end{definition}

\begin{remark}
    Note that if $M$ is a symmetric matroid, then $B(\pi(M))=\emptyset$ if and only if $B(M)=\{\{i,i^*\}\}$ for some $i\in [n]$. The sufficiency is clear, now suppose that $B(\pi(M))=\emptyset$ and $B(M)\neq\{\{i,i^*\}\}$, then there should exist $\{j,j^*\}\in B(M)$ with $j\neq i$, but this forces one of the other diagonals to appear (see the proof of Lemma \ref{lemma4.4a}), which is a contradiction.
    \end{remark} 

 Recall that the degree of $M$ is the cardinality of $A(M)$.

\begin{lemma}\label{lemma4.4a}
Let $M_1,M_2$ be two matroids of degree at least $2$ such that
$\pi(M_1)=\pi(M_2)$. Then, there exists a matroid $M$ such that:
\begin{itemize}
\item[{\rm (i)}] $\pi(M)=\pi(M_1)=\pi(M_2)$.
\item[{\rm (ii)}] If $\mu_i$ denotes the composition of merging some bags, then 
$\mu_1(M)=M_1$ and $\mu_2(M)=M_2$.
\end{itemize}
\end{lemma}

\begin{proof} 
If $A(M_1)\subsetneq A(M_2)$  we must create only the base $\{i,i^*\}$ for $i\in A(M_2)\setminus A(M_1)$ without changing the other bases. We will do this by induction.
If $ A(M_2)\setminus A(M_1)=\{i\}$, and  $M_1$ is given by the  disjoint family $\Pi=\{P_1,\ldots,P_\ell\}$, it follows that, reordering if necessary, there exists $P_i=\{i,i^*\}$ such that $M_2=\sigma_{P_i| i}(M_1)$, which proves the first case.

For
the general case, $ A(M_2)\setminus A(M_1)=\{i_1,\ldots,i_m\}$ with $m\geq2$. Recall that
given $\{a,b,c,d\}\subset E_n$, their three diagonals are 
$$\{\{a,b\},\{c,d\}\},\quad\{\{a,c\},\{b,d\}\},\quad \{\{a,d\},\{b,c\}\}.$$ We denote by $\rho(\{a,b,c,d\})$ the octahedral Grassmann-Pl\"ucker relation involving these three diagonals.
Now, for every pair $i,k$ in the set $\{i_1,\ldots,i_m\}$, we have that $\{i,i^*\},\{k,k^*\}$ are bases of $M_2$, which means that for any $x\in G_{M_2}$, the product $x_{\{i,i^*\}}x_{\{k,k^*\}}$ is nonzero in the octahedral Grassmann-Pl\"ucker relation $\rho(\{i,i^*,k,k^*\})$. Hence, there is at least another diagonal $D$ whose product is nonzero. 
On the other hand, neither $\{i,i^*\}$ nor $\{k,k^*\}$ are bases for $M_1$ (thus, $y_{\{i,i^*\}}y_{\{k,k^*\}}=0$ for any $y\in G_{M_1}$) and the above diagonal $D$ should also appear in $M_1$, thus the Grassmann-Pl\"ucker relation $\rho(\{i,i^*,k,k^*\})$ tells us that actually the third diagonal appears also in $M_1$. Hence all three diagonals must appear in $M_2$.

Let $\Pi(M_1)=\{P_a\}_a$ and $\Pi(M_2)=\{Q_b\}_b$. It follows that $i,i^*\in P_a$, $k,k^*\in P_b$  with $a\neq b$, and $i\in Q_i$, $i^*\in Q_{i^*}$, $k\in Q_k$ and $k^*\in Q_{k^*}$, with $\#\{i,i^*,k,k^*\}=4$.

We claim that $Q_i=\{i\}$ and $Q_{i^*}=\{i^*\}$. Suppose that this is not the case,  say $\ell\in Q_i$ with $\ell\neq i$, then $\{\ell,i^*\}$ is a common basis of both matroids, but then we erase $\{i,i^*\}$  from $M_1$ when merging $Q_i$ and $Q_{i^*}$, which is a contradiction.

Now, if $\ell\in P_a-\{i,i^*\}$, then $\{\ell,i\}$ and $\{\ell,i^*\}$ appear in $B(M_2)$, which is a contradiction since they do not appear in $B(M_1)$. It follows that $P_a=\mu_{i\cup i^*}(Q_i,Q_{i^*})$. 

Therefore, $M_1$ is given by the  disjoint family $\Pi=\{P_1,\ldots,P_\ell\}$, and for every $i\in A(M_2)\setminus A(M_1)$ there exist $P_i=\{i,i^*\}$ such that $$M_2=\sigma_{i_m| i_m}\circ\cdots\circ\sigma_{i_1| i_1}(M_1).$$
This concludes the proof in the case when $A(M_1)\subsetneq A(M_2)$ and 
the case  $A(M_2)\subsetneq A(M_1)$ is similar.

We now suppose that $A(M_1)\not\subset A(M_2)$ and $A(M_2)\not\subset A(M_1)$. 
The simplest sub-case is when $ A(M_2)\setminus A(M_1)=\{i\}$ and $ A(M_1)\setminus A(M_2)=\{j\}$. Since these matroids have degree 2, there is $k\in A(M_2)\cap A(M_1)$ with $i,j\neq k$.

 Since $i\in A(M_2)\setminus A(M_1)$ and $k\in A(M_2)\cap A(M_1)$, we have that $\{i,i^*\},\{k,k^*\}$ are bases of $M_2$, which means that $x_{\{i,i^*\}}x_{\{k,k^*\}}\neq0$ in the octahedral Grassmann-Pl\"ucker relation $\rho(\{i,i^*,k,k^*\})$, so there is at least another diagonal $D$ whose product is nonzero. 
On the other hand, $\{i,i^*\}$ is not a basis for $M_1$ (thus, $y_{\{i,i^*\}}y_{\{k,k^*\}}=0$) and the diagonal $D$ also appears in $M_1$. Hence, the Grassmann-Pl\"ucker relation $\rho(\{i,i^*,k,k^*\})$ implies that the third diagonal appears in $M_1$. Hence all three diagonals must appear in $M_2$.
Setting
 $\Pi(M_1)=\{P_a\}_a$ and $\Pi(M_2)=\{Q_b\}_b$, it follows that $i,i^*\in P_a$, $k\in P_k$ and $k^*\in P_{k^*}$, with $\#\{a,k,k^*\}=3$, and $i\in Q_i$, $i^*\in Q_{i^*}$, $k\in Q_k$ and $k^*\in Q_{k^*}$, with $\#\{i,i^*,k,k^*\}=4$.

Thus $P_a=\{i,i^*\}=\mu_{i\cup i^*}(Q_i=\{i\},Q_{i^*}=\{i^*\})$. Similarly we find $Q_b=\{j,j^*\}=\mu_{j\cup j^*}(P_j=\{j\},P_{j^*}=\{j^*\})$. Then the matroid 
$$\Pi(M_1)-P_a+Q_i+Q_{i^*}=\Pi(M_2)-Q_b+P_j+P_{j^*}$$ is a common refinement of $M_1$ and $M_2$.

The next simplest case is when $ A(M_2)\setminus A(M_1)=\{i,j\}$ and $ A(M_1)\setminus A(M_2)=\{k,\ell\}$ and $A(M_2)\cap A(M_1)=\emptyset$. We argue similarly: since $i,j\in A(M_2)\setminus A(M_1)$ means that $x_{\{i,i^*\}}x_{\{j,j^*\}}\neq0$ in the octahedral Pl\"ucker relation $\rho(\{i,i^*,j,j^*\})$, then there is at least another diagonal $D$ whose product is nonzero. 
On the other hand, these are not bases of $M_1$, (thus, $y_{\{i,i^*\}}y_{\{j,j^*\}}=0$) and the diagonal $D$ also appears in $M_1$, thus the Grassmann-Pl\"ucker relation $\rho(\{i,i^*,j,j^*\})$ implies that the third diagonal appears in $M_1$. Hence all three diagonals must appear in $M_2$. 
It follows that $i,i^*\in P_a$, $j,j^*\in P_b$  with $\#\{a,b\}=2$, and $i\in Q_i$, $i^*\in Q_{i^*}$, $j\in Q_j$ and $j^*\in Q_{j^*}$, with $\#\{i,i^*,j,j^*\}=4$. The general case is obtained
iterating this argument.
\end{proof} 

If $\Delta_N$ is a symplectic matroid polytope which is representable over $\mathbb{C}$, recall that $\pi^{-1}(\Delta_N)=\{\Delta_M\::\:\pi(\Delta_M)=\Delta_N\}$ is non-empty by  Corollary \ref{cor:sym_lift}. 
We say that $\Delta_M\in \pi^{-1}(\Delta_N)$ is \textit{maximal} if $\deg(M)$ is maximal. By Lemma \ref{lemma4.4a}, if $\pi^{-1}(\Delta_N)$ has a member of degree at least 2, then it contains a maximal member, which we denote by $\max(N)$.

\begin{theorem}[Irreducibility of the symplectic thin Schubert cells]\label{theorem3.15}
\ \\ 
\noindent{\rm (1)} For any $N\in M^2_{BC_{n}}(\mathbb{C})$ the corresponding thin Schubert cell $\text{\rm SpG}_N=\{x\in \text{\rm SpG}(2,2n)\::\:M(x))=N\}$   can be written as $$\text{\rm SpG}_N=\bigsqcup_M \left(G_M\cap {\mathcal H}\right)$$ for some $S_{2n}$-matroids $M$.  
From $\text{\rm SpG}(2,2n)=G(2,2n)\cap {\mathcal H}$
we obtain a decomposition, analogue of \eqref{eq1.0},
\begin{equation}\label{eq5.1a}
\text{\rm SpG}(2,2n)=\bigsqcup_{N\in M^2_{BC_{n}}({\mathbb C})}\text{\rm SpG}_N
\end{equation}
into a disjoint union of locally closed sets.\smallskip


\noindent{\rm (2)} There exist precisely 
one $M$ among them that satisfies
\[
\overline{\text{\rm SpG}_N}=\overline{G_M\cap {\mathcal H}}.
\]
In particular $\text{\rm SpG}_N$ is irreducible.
\end{theorem}

\begin{proof}
    Part (1) follows from $\operatorname{Gr}(2,2n)=\bigsqcup_M G_M$ and
    $\operatorname{SpG}(2,2n)=\operatorname{Gr}(2,2n)\cap\mathcal H$ together with the fact that each $G_M\cap\mathcal{H}$ is either completely contained in $\operatorname{SpG}_N$ or does not intersect it.
   For part (2),    by Lemma \ref{lemma4.4a} there exist a maximal matroid $M=([2n],P)$ of degree zero or $\ge 2$ such that $M\cap J_n^2=N$. 
    Since $\overline{G_U}\subset\overline{G_M}$ whenever $U\le M$ and  the equation in (1), it follows that
$\overline{\operatorname{SpG}}_N = \overline{G_M\cap\mathcal H}$ for such $M$.
\end{proof}

We now prove one of our main results. 

\begin{theorem}[The set of $\mathbb{C}$-representable symplectic matroids of rank $2$]
\label{thm4.8}
    If $N$ is a symplectic matroid of rank $2$ over $E_n=[n]\cup [n^*]$, then $N$ is representable over ${\mathbb C}$ if and only if $N$ there exists an $S_{2n}$-matroid $M$ of degree $d\neq1$ such that $\pi(M)=N$.
\end{theorem}
\begin{proof}
    If $N$ is representable, then $\emptyset\neq \operatorname{SpG}_N\subseteq \operatorname{SpG}(2,2n)=G(2,2n)\cap{\mathcal H}$. Thus, there exists $x\in G(2,2n)\cap{\mathcal H}$ such that $x\in \operatorname{SpG}_N$, and   if $x\in G_{M}$, then $\pi(M)=N$ by Proposition \ref{pro4.2} and $\deg (M)\neq 1$ by Lemma \ref{lem:vacuum}.

Conversely, if $\pi(M)=N$ for some $S_{2n}$-matroid $M$ of degree $d\neq1$, then $G_M\cap{\mathcal H}\neq\emptyset$ for some symmetric matroid by Corollary \ref{cor:sym_lift} and $N=M(x)$ for any $x\in G_M\cap{\mathcal H}$.
\end{proof}

\begin{definition}
    Let $N$ be a $\mathbb{C}$-representable symplectic matroids of rank $2$. We say that a symmetric matroid $M$ is a {\it symmetric lifting} of $N$ if $\pi(M)=N$.
\end{definition}

\begin{example}\label{ejemplo:stab}
Assume $N=\{\{i,i^*\}, \{j,j^*\}\}$ and $M=\{\{i\}, \{i^*\}, \{j\}, \{j^*\}\}$ are two $S_{2n}^2$ matroids, with $i<j<j^*<i^*$. Then 
\[
B(N)=\{\{i,j\}, \{i,j^*\}, \{j,i^*\}, \{j^*,i^*\}\},
\] 
while
\[
B(M)=B(N)\cup\{\{i,i^*\}, \{j,j^*\}\}.
\]

It is not difficult to see that
\[
\operatorname{Stab}_{T'}(N) = \left\{ \diag(\lambda_1,\dots,\lambda_n,\lambda_{n^*},\dots, \lambda_{1^*}\} \, |\, \lambda_i\lambda_j=\lambda_i\lambda_{j^*} = \lambda_j\lambda_{i^*} = \lambda_{j^*}\lambda_{i^*}  \right\}
\] 
i.e, 
\[
\operatorname{Stab}_{T'}(N) = \left\{ \diag(\lambda_1,\dots,\lambda_n,\lambda_{n^*},\dots, \lambda_{1^*}\} \, |\, \lambda_i=\lambda_{i^*}\quad \text{and}\quad \lambda_j=\lambda_{j^*}\right\}.
\] 
In particular 
\[
\operatorname{Stab}_{T}(N) = \left\{ \diag(\lambda_1,\dots,\lambda_n,\lambda_{n}^{-1},\dots, \lambda_{1}^{-1}\} \, |\, \lambda_i=\pm 1\quad\text{and}\quad \lambda_j=\pm 1
\right\}.
\]

On the other hand
\[
\operatorname{Stab}_{T'}(M) = \left\{ \diag(\lambda_1,\dots,\lambda_n,\lambda_{n^*},\dots, \lambda_{1^*}\} \, |\, 
\quad \text{equation \ref{eq:2} is satisfied }
\right\}
\] 
where
\begin{equation}\label{eq:2}
\lambda_i\lambda_j=\lambda_i\lambda_{j^*} = \lambda_j\lambda_{i^*} = \lambda_{j^*}\lambda_{i^*} = \lambda_i\lambda_{i^*} = \lambda_j\lambda_{j^*},
\end{equation}
i.e, 
\[
\operatorname{Stab}_{T'}(M) = \left\{ \diag(\lambda_1,\dots,\lambda_n,\lambda_{n^*},\dots, \lambda_{1^*}\} \, |\, \lambda_i=\lambda_{i^*}=\lambda_j=\lambda_{j^*}\right\},
\] 
In particular 
\[
\operatorname{Stab}_{T}(M) = \left\{ \diag(\lambda_1,\dots,\lambda_n,\lambda_{n}^{-1},\dots, \lambda_{1}^{-1}\} \, |\, \lambda_i= \lambda_j=\pm 1
\right\}.
\]
Observe that in this case, $\operatorname{Stab}_T(M)<\operatorname{Stab}_T(N)$ is a subgroup of index 2, so that both have the same dimension and one has
an exact sequence
\[
0\to \operatorname{Stab}_T(N)/\operatorname{Stab}_T(M)\to T/\operatorname{Stab}_T(M) \to T/\operatorname{Stab}_T(N)\to 0
\]
\end{example}

In fact one has the following lemma:

\begin{lemma}\label{simplecticStab}
Assume $N$ is a $BC_n^2$ matroid which admits a degree zero lifting $M_0$ to an $S_{2n}^2$ matroid as well as a degree $r\ge 2$ lifting $M_r$ to an $S_{2n}^2$ matroid. Then $\operatorname{Stab}_T(M_r)<\operatorname{Stab}_T(M_0)$ is a subgroup of index at most two. 
\end{lemma}

\begin{proof}
\begin{itemize}
\item[I.] If $N=\{\{i,i^*\}, \{j,j^*\}\}$, then it is Example \ref{ejemplo:stab} and there is nothing to prove. \\
\item[II.] Assume $A(M) = \{j_1,\dots, j_s\}$ and
$N=\{ Q_1, \dots , Q_r\}$, with $r\ge 3$, $r\ge s$, $s\ge 2$ and $Q_\ell=\{j_\ell, j_\ell^*\}$ for $1\le \ell\le r$. Then
$M=\sigma|_{j_s,j_s}\circ\dots\circ\sigma|_{j_1,j_1}(N)$,
\begin{align*}
B(N)  =  \left\{ \{j_\ell, j_k\}, \{j_\ell, j_k^*\}, \{j_k, j_\ell^*\}, \{j_k^*, j_\ell^*\} \,|\, 1\le j<k\le r
\right\}
\end{align*}
while
\[
B(M)=B(N)\cup\left\{ \{j_\ell,j_\ell^*\}\, |\, 1\le \ell\le s\right\}.\hskip3.3cm{} \text{ }
\]
We claim that $\diag(\lambda_1,\dots, \lambda_n,\lambda_n^*,\dots,\lambda_1^*)\in \operatorname{Stab}_{T'}(N)$ if and only if
\begin{equation}\label{eq:stabT'}
\lambda_{j_1}=\lambda_{j_1^*}=\dots = \lambda_{j_r}=\lambda_{j_r^*}
\end{equation}
which implies that $\diag(\lambda_1,\dots, \lambda_n,\lambda_n^{-1},\dots,\lambda_1^{-1})\in \operatorname{Stab}_{T}(N)$ if and only if
\begin{equation}\label{eq:stabT}
\lambda_{j_1}=\lambda_{j_1^{-1}}=\dots = \lambda_{j_r}=\lambda_{j_r^{-1}}=\pm 1
\end{equation}

To simplify notation, we will write down the proof of Equation \eqref{eq:stabT'} when $r=3$, the general case is proven exactly the same way.

When $r=3$, $\diag(\lambda_1,\dots, \lambda_n,\lambda_n^*,\dots,\lambda_1^*)\in \operatorname{Stab}_{T'}(N)$ if and only if
\begin{align*}
\lambda_{j_1}\lambda_{j_2}  & = \lambda_{j_1}\lambda_{j_2^*} = \lambda_{j_1}\lambda_{j_3}  =\lambda_{j_1}\lambda_{j_3^*}  \cr
& = \lambda_{j_1^*}\lambda_{j_2}  =\lambda_{j_1^*}\lambda_{j_2^*} = \lambda_{j_1^*}\lambda_{j_3}  =\lambda_{j_1^*}\lambda_{j_3^*}  \cr
& = \lambda_{j_2}\lambda_{j_3}  = \lambda_{j_2}\lambda_{j_3^*} = \lambda_{j_2^*}\lambda_{j_3}  =\lambda_{j_2^*}\lambda_{j_3^*} 
\end{align*}
which implies (and in fact is equivalent to)
\[
\lambda_{j_2}=\lambda_{j_2^*}=\lambda_{j_3} = \lambda_{j_3^*}=\lambda_{j_1}=\lambda_{j_1^*},
\]
thus proving Equation \eqref{eq:stabT'}.
In particular,  $\diag(\lambda_1,\dots, \lambda_n,\lambda_n^{-1},\dots,\lambda_1^{-1})\in \operatorname{Stab}_{T}(N)$ if and only if
\[
\lambda_{j_1}=\lambda_{j_1^{-1}}=\dots = \lambda_{j_r} = \lambda_{j_r^{-1}} =\pm 1.
\]
Observe that the additional conditions 
\[
\lambda_{j_1}\lambda_{j_1^*} = \dots = \lambda_{j_s}\lambda_{j_s^*} = \lambda_{j_1}\lambda_{j_2}=\text{ etc.},
\] 
impose to
a matrix $\diag(\lambda_1,\dots, \lambda_n,\lambda_n^*,\dots,\lambda_1^*)$ in order to belong to $\operatorname{Stab}_{T}(M)$ are
automatically satisfied if Equation (\ref{eq:stabT'}) is true, therefore in this case one has $\operatorname{Stab}_{T}(N)=\operatorname{Stab}_T(M)$.
\item[III.] Assume $A(M) = \{j_1,\dots, j_r\}$ and $N=\{P_1,\dots,P_s,Q_1,\dots,Q_r\}$, with $r\ge 2$, 
$s\ge 1$ and $Q_\ell=\{j_\ell, j_\ell^*\}$ for $1\le \ell\le r$. Then
\[
M=\sigma_{j_r|j_r}\circ\dots\circ\sigma_{j_1|j_1}(N),
\] 
\begin{align*}
B(N) & =\left\{ \{m,n\}\, |\, m\in P_a, n\in P_b\quad\text{for some } 1\le a,b \le s; a\ne b\right\}\cr
& \cup \left\{\{t,j\} \, |\, t\in P_k, j\in A(M)\quad\text{for some } 1\le k\le s \right\} \cr
& \cup \left\{\{t,j^*\} \, |\, t\in P_k, j\in A(M)\quad\text{for some } 1\le k\le s\right\} \cr
& \cup \left\{ \{j_a, j_b\} \,|\, \text{for } 1\le a<b\le r\right\}  \cup \left\{ \{j_a, j_b^*\} \,|\, \text{for } 1\le a,b\le r, a\ne b\right\} \cr
& \cup \left\{ \{j_b^*, j_a^*\} \,|\, \text{for } 1\le a<b\le r\right\}
\end{align*}
and
\[
B(M)=B(N)\cup\left\{\{j,j^*\}\,|\, j\in A(M)\right\}.
\]

We claim that $\diag(\lambda_1,\dots,\lambda_n,\lambda_{n^*}, \dots, \lambda_{1^*})\in \operatorname{Stab}_{T'}(N)$ if and only if
\begin{equation}\label{eq:stabT'gen}
\lambda_t = \lambda_{j_\ell} = \lambda_{j_\ell^*}
\end{equation}
for all $t\in P_k$, for all $1\le k\le s$ and for all $1\le \ell\le r$. In particular, this implies that $\diag(\lambda_1,\dots,\lambda_n,\lambda_{n^*}, \dots, \lambda_{1^*})\in \operatorname{Stab}_{T}(N)$ if and only if
\begin{equation}\label{eq:stabTgen} 
\lambda_t = \lambda_{j_\ell} = \lambda_{j_\ell^*} =\pm 1.
\end{equation}
for all $t\in P_k$, for all $1\le k\le s$ and for all $1\le \ell\le r$.
\\

Once more, the additional conditions imposed to be in $\operatorname{Stab}_{T'}(M)$ are automatically satisfied if Equation \eqref{eq:stabT'gen} is satisfied and therefore
in this case one also has $\operatorname{Stab}_{T}(M)=\operatorname{Stab}_{T}(N)$.

Again, to simplify notation we will write down the proof when $s=1$, $r=2$ and $P_1=\{t\}$, and the general case is proven exactly the same way.

Under this simplified assumption, $\diag(\lambda_1,\dots,\lambda_n,\lambda_{n^*}, \dots, \lambda_{1^*})\in \operatorname{Stab}_{T'}(N)$ if and only if
\[
\lambda_t\lambda_{j_1} = \lambda_t\lambda_{j_1^*}= \lambda_t\lambda_{j_2} = \lambda_t\lambda_{j_2^*} =\lambda_{j_1}\lambda_{j_2} =\lambda_{j_1}\lambda_{j_2^*} = \lambda_{j_1^*}\lambda_{j_2} =\lambda_{j_1^*}\lambda_{j_2^*} 
\]
Or equivalently, if and only if
\[
\lambda_{j_1}= \lambda_{j_1^*} = \lambda_{j_2} = \lambda_{j_2^*} = \lambda_t
\]
as claimed, hence proving Equation \eqref{eq:stabT'gen} and henceforth also Equation \eqref{eq:stabTgen} and the lemma. \qedhere
\end{itemize}
\end{proof}

\begin{corollary}\label{cor4.14}
For any matroid $N\in BC_n^2$ and any $x,y\in \operatorname{SpG}_N$ we have $\dim T\cdot x = \dim T\cdot y$, in particular the action of $T$ on $\operatorname{SpG}_N$
is closed.\hskip3cm{$\square$}
\end{corollary}

\begin{corollary}\label{cor4.15}
For any matroid $N\in BC_n^2$, there exist an irreducible, universal geometric quotient $\mathcal{S}pG_N=\operatorname{SpG}_N/T$, which may be a non-separated scheme. 
\end{corollary}

\begin{proof}
By Sumihiro \cite{S} and Corollary \ref{cor4.14}, any $x\in \operatorname{SpG}_N$ is pre-stable (in the sense of Mumford \cite{M}), therefore the quotient $\mathcal{S}pG_N=\operatorname{SpG}_N/T$ exists and is an irreducible, universal geometric quotient. \hskip3cm{}
\end{proof}

Notation: If $M\in M_{S_{2n}}^2$ is the unique symmetric lifting of $N\in B_{n}^2$ satisfying
\[
\overline{\operatorname{SpG}_N} = \overline{G_M\cap \mathcal H},
\]
as in Theorem \ref{theorem3.15}, we will write $T_N:= T/\operatorname{Stab}T(M)$.

Let $M$ be a symmetric matroid such that $B(M)\neq \{\{i,i^*\}\}$ for some $i\in [n]$. Then  its symplectic projection $\pi (M)$ (see Definition \ref{def:sympl_proj}) is a non-empty set. Furthermore, we know by Lemma \ref{lem:vacuum} that if $\deg(M)\neq1$, then $\pi (M)$ is a $\mathbb{C}$-representable symplectic matroid. 

Conversely, if $N$ is a $\mathbb{C}$-representable symplectic matroid of rank $2$ over $E_n=[n]\cup [n^*]$, we have three cases:
\begin{enumerate}
    \item $N$ admits only a symmetric lifting of degree $d<1$,
    \item $N$ only admits  symmetric liftings of degree $d>1$,
    \item $N$ admits a symmetric lifting of degree $d<1$ and one of degree $d>1$.
\end{enumerate}

\begin{definition}
\label{def:trichotomy}
    Let $N$ be a $\mathbb{C}$-representable symplectic matroids of rank $2$. We define its set of normal (canonical) symmetric liftings $M_{{S_{2n}}}^2\downarrow N$ as follows: 
    \begin{enumerate}
    \item If $N$ admits only a symmetric lifting of degree $d<1$, then $M_{{S_{2n}}}^2\downarrow N=\{N\}$,
    \item If $N$ only admits symmetric liftings of degree $d>1$, then $M_{{S_{2n}}}^2\downarrow N=\{\max(N)\}$, where $\max(N)$ is the maximal symmetric matroid projecting to $N$ as in Theorem \ref{theorem3.15},
    \item If $N$ admits a symmetric lifting of degree $d<1$ and one of degree $d>1$, then $M_{{S_{2n}}}^2\downarrow N=\{N, \max(N)\}$.
\end{enumerate}
\end{definition}

\begin{remark}
The set of normal symmetric liftings $M_{{S_{2n}}}^2\downarrow N$ of the $\mathbb{C}$-representable symplectic matroids of rank two  gives a novel way, different from Definition \ref{Def_BC_nM}, for expressing these objects as symmetric matroids, that is, as pairs $(E_n,P)$ where $P$ is  a non-intersecting family of $E_n$. 
In particular, this new description should help 
 to describe the set $\tfrac{M^2_{BC_{n}}({\mathbb C})}{BC_{n}}$ of equivalence classes of $\mathbb{C}$-representable $BC_n$-matroids of rank 2, which is important for the understanding of the cycle map (see Corollary \ref{cor:factorization_cycle_map}). However, we do not know what happens for the set $M^2_{BC_n}\setminus M^2_{BC_n}(\mathbb{C})$ of symplectic matroids which are not $\mathbb{C}$-representable. The next example is illustrative.
\end{remark}

\begin{example}\label{example4.9a}
 Let $N$ be the  symplectic projection of the symmetric matroid $M=\big([3^*]\cup [3], \{122^*\}\{1^*33^*\}\big)$ of degree 1. Then, $N$  is a symplectic matroid by \cite{BGW} and it does not admit a lifting of degree $d\neq 1$. Therefore, it is not representable.
\end{example}
  
This example  raises the following questions: What happens for the symplectic matroids which are not $\mathbb{C}$-representable? Can all of them be obtained as the symplectic projection $\pi (M)$ of a symmetric matroid?
Assuming that $\deg(M)=1$, is it true that $\pi(M)$ is a symplectic matroid?

\section{$T$-invariant schematic points}\label{sec5}

 For each $N\in M^2_{BC_{n}}({\mathbb C})$, by Theorem \ref{theorem3.15} the symplectic thin Schubert cell $\text{SpG}_N$ is an irreducible $T$-invariant space.  
 By Corollary \ref{cor4.15}
 the universal geometric quotient $\text{Sp}\mathcal{G}_N=\text{SpG}_N/T$ exists and it is irreducible, and by Corollary \ref{cor4.14} all fibers have the same dimension, hence we can compute the dimension of the symplectic thin Schubert cell $\operatorname{SpG}_N$. 
 
 \begin{definition}
     Let $M=(E,\Pi)$ be a matroid with $\Pi=\{P_1,\ldots,P_\ell\}$. The length of $M$ is $\ell(M)=\ell$, and the weight of $M$ is $w(M)=\#(P_1\cup\cdots\cup P_\ell)$.
 \end{definition}
 


\begin{theorem}
Let $N$ be a $\mathbb{C}$-representable symplectic matroids of rank $2$. 
Then, $\dim \text{\rm SpG}_N=\dim T - \dim \operatorname{Stab_T(M)}+\dim \text{\rm Sp}\mathcal{G}_N$, where $M=(E_{n},\Pi)$ is a symmetric lifting of $N$ of maximal degree. Moreover, we have 
\begin{equation*}
   \dim \text{\rm SpG}_N=\begin{cases}
    0+0,& \text{if } \ell(P)=2,\\
    [w(M)-1]+[\ell(M)-3],&\text{if }\ell(M)>2, \deg(M)<1,\\
    [w(M)-\deg(M)]+[\ell(M)+\deg(M)-5],&\text{if }\ell(M)>2, \deg(M)>1,\\
   \end{cases}
\end{equation*}
\end{theorem}
\begin{proof}
If $M$ is a symmetric lifting of $N$ of maximal degree, then it satisfies $\overline{\text{SpG}_N}=\overline{G_{M}|_H}$, thus we have $\dim\text{SpG}_N=\dim G_{M}$ if $\deg M<1$, and $\text{SpG}_N=\dim G_{M}-1$ if deg $M>1$. 


On the other hand, by \cite[Theorem 4.10]{EFG} we have $\dim G_{M}=w(M)+\ell(M)-4$, and if $\ell(M)>2$, then $\dim G_{M}=[w(M)-1]+[\ell(M)-3]$ where the first summand is the dimension of the fiber and the second is the dimension of the basis in the torus bundle decomposition of $G_M$. 

If the degree is 0, then the thin Schubert cell is contained in the hyperplane, so the dimension of the whole cell and the the dimension of the symplectic cell remain the same.

If the degree is not 0,  the dimension of the thin Schubert cell  $\text{SpG}_N$ is $w(P)-1-(\deg(P)-1)$ which is the dimension of the projection of the polytope. And since the total dimension is $w(P)+\ell(P)-5$, the dimension of the geometric quotient $ \text{\rm Sp}\mathcal{G}_N$ is $\ell(M)+\deg(M)-5$.
\end{proof}

Given  $\mathcal{Y}\subset \text{Sp}\mathcal{G}_{N}$ a subvariety, that is a closed subscheme which is reduced and irreducible, we define 
\begin{equation}    \mathbb{V}_N(\mathcal{Y})=\overline{\mathcal{Y}\times_{\text{Sp}\mathcal{G}_{N}}\text{SpG}_{N}}.
\end{equation}
Thus, $\mathbb{V}_N(\mathcal{Y})\subset \text{\rm SpG}(2,2n)$ is a $T$-invariant subvariety. The converse is also true:

\begin{theorem}\label{thm3.18}
If $Y\subset \text{\rm SpG}(2,2n)$ is a $T$-invariant subvariety, then $Y=\mathbb{V}_{M(Y)}(\mathcal{Y})$ for unique $M(Y)\in M^2_{BC_{n}}(\mathbb{C})$ and a subvariety $\mathcal{Y}\subset \text{\rm Sp}\mathcal{G}_{M(Y)}$.
\end{theorem}
\begin{proof}
If $Y\subset \text{SpG}(2,2n)$ is a $T$-invariant subvariety, by Theorem \ref{theorem3.15} $$Y=Y\cap \text{SpG}(2,2n)=\bigsqcup_{M\in M^2_{BC_{n}}(\mathbb{C})}(Y\cap \text{SpG}_M),$$
and $Y=\overline{Y}=\overline{\bigsqcup_{M}(Y\cap \text{SpG}_M)}=\bigcup_{M}\overline{(Y\cap \text{SpG}_M)}$ since it is closed. Hence, there exists a unique $M(Y)\in M^2_{BC_{n}}(\mathbb{C})$ such that $Y=\overline{Y\cap \text{SpG}_{M(Y)}}$ because $Y$ is irreducible. Now, $Y\cap \text{SpG}_{M(Y)}$
is open dense in $Y$ and  
by the proof of Lemma \ref{simplecticStab}) 
the geometric quotient $(Y\cap \text{SpG}_{M(Y)})/T=:\mathcal{Y}\subset \text{Sp}\mathcal{G}_{M(Y)}$ exists and thus 
\begin{equation}
Y=\overline{\mathcal{Y}\times_{\text{Sp}\mathcal{G}_{M(Y)}}\text{SpG}_{M(Y)}}=\mathbb{V}_{M(Y)}(\mathcal{Y}).\qedhere
\end{equation}
\end{proof}


\begin{example}[Case $n=2$]
\label{example3.6}
Using Theorem \ref{thm3.18}, we will compute all the $T$-invariant subvarieties of $\text{Sp}G(2,4)$. By Table \ref{table3.1} of Example  \ref{example3.3} we have fifteen $BC_n$-matroids. Fourteen of them are torus orbits having trivial quotient spaces $\text{Sp}\mathcal{G}_M$. The exceptional generic matroid $N=\{x_{12},x_{1^*2}, x_{1 2^*}, x_{1^*2^*}\}$, whose quotient $\text{Sp}\mathcal{G}_N$ is a curve,  and any closed point $\{*\}\in  \text{Sp}\mathcal{G}_N$ of this curve
corresponds to an orbit. We have one more $T$-invariant subvariety which is not an orbit, coming from chosing the generic point in $\text{Sp}\mathcal{G}_N$  for which we have $\mathbb{V}_M(\text{Sp}\mathcal{G}_M)=\text{Sp}G(2,4)$. 

We now describe the homology class of these sixteen $T$-invariant subvarieties in $H_*(\text{SpG}(2,4),\mathbb{Z})$. By Proposition \ref{Prop_homology}, we have that  $H_i(\text{SpG}(2,4),\mathbb{Z})\simeq\mathbb{Z}$ for $i=0,2,4,6$ and $\{0\}$ otherwise. 
The $16=15+1$ types of $T$-invariant subvarieties fall into at most  6 different homology classes, namely 
\begin{enumerate}
    \item real dimension 0: 
    all the four 0-dimensional orbits are fixed points $\{x_{ij}\}$, for $\{i,j\}$ admissible, and have homology class $[*]$,
    \item real dimension 2: 
    \begin{enumerate}
        \item all the four 1-dimensional closure of the matroidal orbits $\overline{G_M}$ for  $M=\{x_{ij}, x_{i^*j}\}$ or $M=\{x_{ij}, x_{ij^*}\}$ with admissible coordinates have homology class $[\mathbb{P}^1]$,
        \item The two 1-dimensional closure of the matroidal orbit $\overline{G_M}$ for $M=\{x_{ij}, x_{i^*j^*}\}$ have homology class $[H]^2$, which is  of the form $2\cdot[\mathbb{P}^1]$, since the degree of the intersection of the plane and the quadric is $2$ by Bezout's theorem.
    \end{enumerate}
    \item real dimension 4: 
    \begin{enumerate}
        \item all the four 2-orbits have the hyperplane class  $[H]$ for the matroids $M=\binom{J_2^2}{2}\setminus \{x_{ij}\}$, and
        \item the orbit of a closed point in the generic stratum is $2[H]$,
    \end{enumerate}
    \item real dimension 6: the homology class $[\text{SpG}(2,4)]$.
\end{enumerate}
 \end{example}
 
Consider the cycle map $h:\text{\rm SpG}(2,2n)^{T}\xrightarrow[]{}H_*(\text{\rm SpG}(2,2n))$ that assigns to a $T$-invariant subvariety its homology class. We have a symplectic analogue of \cite[Theorem 3.2]{EFG}:

\begin{corollary}
\label{cor:factorization_cycle_map}
Considering $\mathbb{V}_M$ as the functor of points $\text{\rm Sp}\mathcal{G}_M\to \text{\rm SpG}(2,2n)^{T}$ we have
\begin{equation*}\text{\rm SpG}(2,2n)^{T}=\bigsqcup_{M\in M^2_{BC_{n}}({\mathbb C})} \mathbb{V}_{M}(\text{\rm Sp}\mathcal{G}_M). 
\end{equation*}
\end{corollary}

Moreover, the cycle map $h:\text{\rm SpG}(2,2n)^{T}\xrightarrow[]{}H_*(\text{\rm SpG}(2,2n))$ factors through the projection 

\begin{equation}
    q:\bigsqcup_{M\in M^2_{BC_{n}}({\mathbb C})} \mathbb{V}_{M}(\text{\rm Sp}\mathcal{G}_M)\xrightarrow[]{}\bigsqcup_{[M]\in \tfrac{M^2_{BC_{n}}({\mathbb C})}{BC_{n}}} \mathbb{V}_M(\text{\rm Sp}\mathcal{G}_{[M]}).
\end{equation}

Here $M^2_{BC_{n}}({\mathbb C})/BC_{n}$ is the set of isomorphism classes of $\mathbb{C}$-representable symplectic matroids of rank 2.

\end{document}